\documentclass[12pt,a4paper]{article}
\usepackage{amsfonts}
\usepackage{amssymb}
\usepackage{amsmath}

\usepackage{pifont}

\newtheorem{theorem}{Theorem}[section]
\newtheorem{definition}[theorem]{Definition}
\newtheorem{lemma}[theorem]{Lemma}
\newtheorem{Corollary}[theorem]{Corollary}

\newenvironment{proof}{{\bf Proof.  }}{$\square$}

\begin{document}
\title{Composition-Diamond Lemma for Non-associative Algebras over a Commutative Algebra\footnote{Supported by the
NNSF of China (No.10771077, 10911120389) and the NSF of Guangdong
Province (No. 06025062).}}

\author{
 Yuqun
Chen, Jing Li and Mingjun Zeng \  \\
{\small \ School of Mathematical Sciences, South China Normal
University}\\
{\small Guangzhou 510631, P. R. China}\\
{\small  yqchen@scnu.edu.cn}\\
{\small yulin\_jj@yahoo.com.cn}
\\
{\small dearmj@126.com}}

\date{}
\maketitle

\maketitle \noindent\textbf{Abstract:} We establish the
Composition-Diamond lemma for non-associative algebras over a free
commutative algebra. As an application, we prove that  every
countably generated non-associative algebra over an arbitrary
commutative algebra $K$ can be embedded into a two-generated
non-associative algebra over $K$.


\noindent {\bf AMS} Mathematics Subject Classification(2000): 16S15,
13P10, 17Dxx, 13Axx

\section{Introduction}
 Gr\"{o}bner bases and Gr\"{o}bner-Shirshov bases theories were
invented independently by A.I. Shirshov \cite{Shir3} for non-associative algebras
and commutative (anti-commutative) non-associative algebras
\cite{Sh62a}, for Lie algebras (explicitly) and associative algebras
(implicitly) \cite{Sh62b}, for infinite series algebras (both formal
and convergent) by H. Hironaka \cite{Hi64} and for polynomial
algebras by B. Buchberger (first publication in \cite{Bu70}).
Gr\"{o}bner bases and Gr\"{o}bner-Shirshov bases theories have been
proved to be very useful in different branches of mathematics,
including commutative algebra and combinatorial algebra, see, for
example, the books
 \cite{AL, BKu94, BuCL, BuW, CLO, Ei}, the papers \cite{Be78, Bo72,Bo76,BoC1,CM},
 and the surveys \cite{BC, BFKK00, BK03, BK05}.

It is well known that every countably generated non-associative
algebra over a field $k$ can be embedded into a two-generated
non-associative algebra over $k$. This result follows from
Gr\"{o}bner-Shirshov bases theory for non-associative algebras by
A.I. Shirshov \cite{Sh62a}.

Composition-Diamond lemmas for associative algebras over a
polynomial algebra is established by A.A. Mikhalev and A.A. Zolotykh
\cite{MZ}, for associative algebras over an associative algebra by
L.A. Bokut, Yuqun Chen and Yongshan Chen \cite{BCC08}, for Lie
algebras over a polynomial algebra by L.A. Bokut, Yuqun Chen and
Yongshan Chen \cite{BCC09}. In this paper, we establish the
Composition-Diamond lemma for non-associative algebras over a
polynomial algebra. As an application, we prove that every countably
generated non-associative algebra over an arbitrary commutative
algebra $K$ can be embedded into a two-generated non-associative
algebra over $K$, in particular, this result holds if $K$ is a free
commutative algebra.

\section{Composition-Diamond lemma for non-associative algebras over a commutative algebra}
Let $k$ be a field, $K$ a commutative associative $k-$algebra with
unit, $X$ a set and ${K}(X)$ the free non-associative algebra over
$K$ generated by $X$.

Let $[Y]$ denote the free abelian monoid generated by $Y$, $X^{*}$
the free monoid generated by $X$ and $X^{**}$ the set of all
non-associative words in $X$. Denote by
$$
N=[Y]X^{**}=\{u=u^{Y}u^{X}|u^{Y}\in [Y],u^{X}\in X^{**}\}.
$$

Let $kN$ be a $k$- linear space spanned by $N$. For any $u=u^Yu^X,\
      v=v^Yv^X\in N$, we define the multiplication of the words as
follows
$$
uv=u^Yv^Yu^Xv^X\in N.
$$
It is clear that $kN$ is  the free non-associative $k[Y]$-algebra
generated by $X$. Such an algebra is denoted by ${k[Y]}(X)$, i.e.,
$kN={k[Y]}(X)$. Clearly,
$$
{k[Y]}(X)={k[Y]}\otimes k(X).
$$

Now, we order the set $N=[Y]X^{**}$.

Let $>$ be a total ordering on $X^{**}$. Then $>$ is called monomial
if
$$
(\forall u,v,w\in X^{**}) \ \ u>v\Rightarrow wu>wv \ \mbox{ and }\ \
uw>vw.
$$
For example, the deg-lex ordering  on $X^{**}$ is  monomial: $uv>u_{1}v_1$, if $deg(uv)>deg(u_{1}v_1)$, otherwise $u>u_1$ or $u=u_1, v>v_1$.
Similarly, we define the monomial ordering on $[Y]$.

Suppose that both $>_{X}$ and $>_{Y}$ are monomial orderings on
$X^{**}$ and $[Y]$, respectively. For any $u=u^Yu^X,v=v^Yv^X\in N$,
define
$$
u>v\Leftrightarrow 
u^X>_{X}v^X \  or \ (u^X=v^X \ and \ u^Y>_{Y}v^Y).
$$
It is obvious that $>$ is a monomial ordering on $N$ in the sense
of
$$
(\forall u,v,w\in [Y]X^{**}) \ \ u>v\Rightarrow wu>wv, \ uw>vw\
\mbox{ and }\ \ w^Yu>w^Yv.
$$
We will use this ordering in this paper.

For any polynomial $f\in k[Y](X)$, $f$ has a unique presentation of
the form
$$
f=\alpha_{\bar{f}}\bar{f}+\sum\alpha_iu_i,
$$
where $\bar{f},u_i\in
[Y]X^{**},\bar{f}>u_i,\alpha_{\bar{f}},\alpha_i\in k.$ $\bar{f}$ is
called the leading term of $f$. $f$ is monic if the coefficient of
$\bar{f}$ is 1.

\ \

Let  $\star \not\in X$. By a $\star$-word we mean any expression in
$[Y](X \cup \{\star\})^{**}$ with only one occurrence of $\star$.
Let $u$ be a $\star$-word and $s\in k[Y](X)$. Then we call $u|_s =
u|_{\star\mapsto s}$ an $s$-word.

It is clear that for $s$-word $u|_s$, we can express $u|_s=u^Y(asb)$
for some $a,b\in X^{*}$.

Since $>$ is monomial on $[Y]X^{**}$, we have following lemma.

\begin{lemma}
 Let $s\in k[Y](X)$ be a non-zero polynomial.
Then for any $s$-word  $u|_s=u^Y(asb)$,
$\overline{u^Y(asb)}=u^Y(a\bar{s}b)$.
\end{lemma}

Now, we give the definition of compositions.

\begin{definition}
Let $f$ and $g$
be monic polynomials of $k[Y](X)$,  $w=w^Yw^X\in [Y]X^{**}$ and
$a,b,c\in X^*$, where $w^Y=L(\bar{f}^Y, \bar{g}^Y)\triangleq L$ and
$L(\bar{f}^Y, \bar{g}^Y)$ is the least common multiple of
$\bar{f}^Y$ and $\bar{g}^Y$ in  $k[Y]$. Then we have the following
compositions.

$1.$ $X$-inclusion

If $w^X=\bar{f}^X=(a(\bar{g}^X)b)$, then
$$
(f,g)_{w}=\frac{L}{ \bar{f}^Y} f-\frac{L}{ \bar{g}^Y}(a(g)b)
$$
is called the composition of $X$-inclusion.

$2.$  $Y$-intersection only

If $|\bar{f}^Y|+|\bar{g}^Y|> |w^Y|$ and
$w^X=(a(\bar{f}^X)b(\bar{g}^X)c)$, then
$$
(f,g)_{w}=\frac{L}{ \bar{f}^Y} (a(f)b(\bar{g}^X)c)-\frac{L}{ \bar{g}^Y} (a(\bar{f}^X)b(g)c)
$$
is called the composition of $Y$-intersection only, where for $u\in
[Y],\ |u|$ means the degree of $u$.

 $w$ is called the
 ambiguity of the composition $(f,g)_w$.
\end{definition}

\noindent{\bf Remark 1}.In the case of $Y$-intersection only in
Definition 2.2, $\bar{f}^X$ and $\bar{g}^X$ are disjoint.

\noindent{\bf Remark 2}. By Lemma 2.1, we have
$w>\overline{(f,g)_w}.$

\noindent{\bf Remark 3}. In Definition 2.2, the compositions of
$f,g$ are the same as the ones in $k(X)$, if $Y=\emptyset$. If this
is the case, we have only composition of $X$-inclusion.

\begin{definition}
 Let $S$ be a monic subset of $k[Y](X)$ and $f,g\in S$. A composition
$(f,g)_w$ is said to be \emph{trivial modulo} $(S,w)$, denoted by
$(f,g)_w\equiv 0 \ \ mod(S,w)$, if
$$
(f,g)_w=\sum_i\alpha_iu_i|_{s_i},
$$
where each $s_i\in S,
  \ \alpha_i\in k, \ u_i|_{{s_i}} \ s_i$-word and $w>u_i|_{\bar{s_i}}$.

Generally, for any $p,q\in k[Y](X),\ p\equiv q \ \ mod (S,w)$ if and
only if $p-q\equiv 0 \ \ mod(S,w).$

 $S$ is called a \emph{Gr\"{o}bner-Shirshov basis} in
$k[Y](X)$ if all compositions of elements in $S$ are trivial modulo
$S$.
\end{definition}

 If a subset $S$ of  $k[Y](X)$ is not a Gr\"{o}bner-Shirshov basis
then one can add to $S$ all nontrivial compositions of polynomials
of $S$ and continue this process repeatedly so that we obtain a
Gr\"{o}bner-Shirshov basis $S^{c}$ that contains $S$. Such process
is called the Shirshov algorithm.

\begin{lemma}
Let $S$ be a
Gr\"{o}bner-Shirshov basis in $k[Y](X)$ and $s_1,s_2\in S$. Let
$u_1|_{s_1},\ u_2|_{s_2}$ be $s_1,s_2$-words respectively. If
$w=u_1|_{\overline{s_1}}=u_2|_{\overline{s_2}}$, then
$u_1|_{s_1}\equiv u_2|_{s_2} \ mod(S,w)$.
\end{lemma}

\textbf{Proof: }
 Clearly, $w^Y=L(\bar{s_1}^Y,\bar{s_2}^Y)\cdot
t=L\cdot t$ for some $t\in [Y]$.

There are three cases to consider.

\emph{Case 1}. $X$-inclusion.

We may assume that $\bar{s_1}^X=(c(\bar{s_2}^X)d)$ for some $c,d\in
X^*$ and $w^X=(a(\bar{s_1}^X)b)=(a(c(\bar{s_2}^X)d)b)$ for some
$a,b\in X^*$. Thus,
\begin{eqnarray*}
u_1|_{s_1}-u_2|_{s_2}&=&\frac{L\cdot t}{\bar{s_1}^Y}(a(s_1)b)-\frac{L\cdot t}{\bar{s_2}^Y}(a(c(s_2)d)b)\\
&=&t\cdot(a(\frac{L}{\bar{s_1}^Y}s_1-\frac{L}{\bar{s_2}^Y}(c(s_2)d))b)\\
&=&t\cdot(a(s_1,s_2)_{w_1}b)\\
&\equiv&0 \ \ \ \ \ \ \ mod(S,w)
\end{eqnarray*}
where $w_1=L\overline{s_1}^X$.

\emph{Case 2}. $Y$-intersection only.

In this case, $w^X=(a(\bar{s_1}^X)b(\bar{s_2}^X)c),\ a,b,c\in X^*$
and then
\begin{eqnarray*}
u_1|_{s_1}-u_2|_{s_2}&=&\frac{L\cdot t}{\bar{s_1}^Y}(a(s_1)b(\bar{s_2}^X)c)-
\frac{L\cdot t}{\bar{s_2}^Y}(a(\bar{s_1}^X)b(s_2)c)\\
&=&t\cdot(s_1,s_2)_{w_1}\\
&\equiv&0 \ \ \ \ \ \ \ mod(S,w)
\end{eqnarray*}
where $w_1=Lw^X$.

\emph{Case 3}. $Y$-disjoint and $X$-disjoint.

In this case, $L=\bar{s_1}^Y\bar{s_2}^Y$ and
$w^X=(a(\bar{s_1}^X)b(\bar{s_2}^X)c),\ a,b,c\in X^*$. We have
\begin{eqnarray*}
u_1|_{s_1}-u_2|_{s_2}&=&\frac{L\cdot t}{\bar{s_1}^Y}(a(s_1)b(\bar{s_2}^X)c)-
\frac{L\cdot t}{\bar{s_2}^Y}(a(\bar{s_1}^X)b(s_2)c)\\
&=&t\cdot(\frac{L}{\bar{s_1}^Y}(a(s_1)b(\bar{s_2}^X)c)-\frac{L}{\bar{s_2}^Y}(a(\bar{s_1}^X)b(s_2)c))\\
&=&t\cdot(\bar{s_2}^Y(a(s_1)b(\bar{s_2}^X)c)-\bar{s_1}^Y(a(\bar{s_1}^X)b(s_2)c))\\
&=&t\cdot((a(s_1)b(\bar{s_2})c)-(a(\bar{s_1})b(s_2)c))\\
&=&t\cdot((a(s_1)b(\bar{s_2})c)-(a(s_1)b(s_2)c)+(a(s_1)b(s_2)c)-(a(\bar{s_1})b(s_2)c))\\
&=&t\cdot((a(s_1-\bar{s_1})b(s_2)c)-(a(s_1)b(s_2-\bar{s_2})c))\\
&\equiv&0 \ \ \ \ \ \ \ mod(S,w)
\end{eqnarray*}
since
$w=(a(\bar{s_1})b(\bar{s_2})c)>\overline{(a(s_1-\bar{s_1})b(s_2)c)}$
and
$w=(a(\bar{s_1})b(\bar{s_2})c)>\overline{(a(s_1)b(s_2-\bar{s_2})c)}$.

This completes the proof.\  \hfill $\square$

\begin{lemma}
Let $S\subseteq k[Y](X)$ with each
$s\in S$ monic and $Irr(S)=\{w \in [Y]X^{**}|w\neq u|_{\bar{s}},\
u|_{{s}} \ \mbox{is an}\ s\mbox{-word}, \ s\in S\}$. Then for any
$f\in k[Y](X)$,
$$
f=\sum_{u_i|_{\overline{s_i}}\leq \bar{f}} \alpha_i
u_i|_{s_i}+\sum_{v_j\leq \bar{f}} \beta_j v_j,
$$
where  $\alpha_i,  \beta_j \in k,\ u_i|_{{s_i}} \ s_i$-word, \
$s_i\in S \ \mbox{and} \ v_j\in Irr(S)$.
\end{lemma}

{\bf Proof.}Let  $f=\sum\limits_{i}\alpha_{i}u_{i}\in k[Y](X)$,
where $0\neq{\alpha_{i}\in{k}}$ and $u_{1}>u_{2}>\cdots$. If
$u_1\in{Irr(S)}$, then let $f_{1}=f-\alpha_{1}u_1$. If
$u_1\not\in{Irr(S)}$, then there exists an $s$-word $u|_s$
 such that $\bar f=u|_{\bar{s}}$. Let
$f_1=f-\alpha_1u|_{{s}}$. In both cases, we have
$\bar{f}>\bar{f_1}$. Then the result follows from the induction on
$\bar{f}$. \  \hfill \ \ $\square$

From the above lemmas, we reach the following theorem:

\begin{theorem}
(Composition-Diamond lemma for  $k[Y](X)$)\label{t1} Let
$S\subseteq k[Y](X)$ with each $s\in S$ monic, $>$ the ordering on
$[Y]X^{**}$ defined as before and $Id(S)$ the ideal of $k[Y](X)$
generated by $S$ as $k[Y]$-algebra. Then the following statements
are equivalent:
\begin{enumerate}
\item[(i)]\ $S$ is a Gr\"obner-Shirshov basis in $k[Y](X)$.
\item[(ii)]\ If \ $0\neq
f\in Id(S)$, then $\overline{f}=u|_{\overline{s}}$ for some $s$-word
$u|_s , \ s\in S$.
\item[(iii)]\
$Irr(S)=\{w \in [Y]X^{**}|w\neq u|_{\bar{s}}, u|_{{s}} \ \mbox{is
an}\ s\mbox{-word}, \ s\in S\}$ is a $k$-linear basis for the factor
algebra $k[Y](X|S)=k[Y](X)/Id(S)$.
\end{enumerate}
\end{theorem}

\textbf{Proof: }
$(i)\Rightarrow (ii)$. Suppose $0\neq f\in Id(S)$.
Then $f=\sum \alpha_i u_i|_{s_i} $ for some $\alpha_i\in k, \
s_i$-word $u_i|_{s_i}, \ s_i\in S$. Let $w_i=u_i|_{\overline{s_i}}$
and $w_1=w_2=\dots=w_l>w_{l+1}\geq \cdots$. We will prove the result
by using induction on $l$ and $w_1$.

If $l=1$, then the result is clear. If $l>1$, then
$w_1=u_1|_{\overline{s_1}}=u_2|_{\overline{s_2}}$. Now, by (i) and
Lemma 2.4, $u_1|_{s_1}\equiv u_2|_{s_2} \ \ mod(S,w_1)$. Thus,
\begin{eqnarray*}
\alpha_1u_1|_{s_1}+\alpha_2u_2|_{s_2}&=&(\alpha_1+\alpha_2)u_1|_{s_1}+\alpha_2(u_2|_{s_2}-u_1|_{s_1})\\
&\equiv&(\alpha_1+\alpha_2)u_1|_{s_1} \ \ \ \ \ \ mod(S,w_1).
\end{eqnarray*}
Therefore, if $\alpha_1+\alpha_2\neq 0$ or $l>2$, then the result
follows from the induction on $l$. For the case $\alpha_1+\alpha_2=
0$ and $l=2$, we use the induction on $w_1$. Now the result follows.

$(ii)\Rightarrow (iii)$. By Lemma 2.5,  $Irr(S)$ generates the
factor algebra. Moreover, if $0\neq h=\sum \beta_j u_j\in Id(S)$,
$u_j\in Irr(S), u_1>u_2>\cdots $ and  $ \beta_1\neq 0$, then
$u_1=\bar{h}=u|_{\bar{s}}$, a contradiction. This shows that
$Irr(S)$ is a $k$-linear basis of the factor algebra.

$(iii)\Rightarrow (i)$. For any $f, \ g\in S$, since $k[Y]S\subseteq
Id(S)$, we have $h=(f,g)_w\in Id(S)$. The result is trivial if
$(f,g)_w=0$. Assume that $(f,g)_w\neq 0$. Then, by Lemma 2.5, (iii)
and by noting that $w>\overline{(f,g)_w}=\bar{h}$, we have
$(f,g)_w\equiv0\ \ mod(S,w)$.

This shows (i).\ \hfill\ \ $\square$

\noindent{\bf Remark}: Theorem 2.6 is the Composition-Diamond lemma
for non-associative algebras when $Y=\emptyset$.

\section{Applications}
 Let $A$ be an arbitrary $K$-algebra and  $A$  be presented  by
generators $X$ and  defining relations $S$
$$
A={K}(X|S).
$$

Let $K$ have a presentation by generators $Y$ and  defining
relations $R$
$$
K=k[Y|R]
$$
as a quotient algebra of the polynomial algebra $k[Y]$ over $k$.

Then with a natural way, as $k[Y]$-algebras, we have an isomorphism
$$
k[Y|R](X|S)\rightarrow k[Y](X|S^l,Rx,x\in X), \sum
(f_i+Id(R))u_i+Id(S)\mapsto\sum f_iu_i+Id(S'),
$$
where $f_i\in k[Y], \ u_i\in X^{**},\ S'=S^l\cup \{gx|g\in R,\ x\in
X\},\ S^l=\{\sum f_iu_i\in k[Y](X)|\sum (f_i+Id(R))u_i\in S \}$.
Then $A$ has an expression
$$
A=k[Y|R](X|S)=k[Y](X|S^l,gx,\ g\in R,x\in X).
$$

\begin{theorem}
Each countably
generated non-associative algebra over an arbitrary commutative
algebra $K$ can be embedded into a two-generated non-associative
algebra over
$K$.
\end{theorem}
\begin{proof}\ Let the notation be as before.
Let $A$ be the non-associative algebra over $K=k[Y|R]$ generated by
$X=\{x_i|i=1,2,\dots\}$. We may assume that $A={k[Y|R]}(X|S)$ is
defined as above. Then $A$ can be presented as $A=k[Y](X|S^l,gx_i,\
g\in R, \ i=1,2,\dots)$. By Shirshov algorithm, we can assume that,
with the deg-lex ordering $>_Y$ on $[Y]$, $R$ is a
Gr\"{o}bner-Shirshov basis in the free commutative algebra $k[Y]$.
Let $>_X$ be the deg-lex ordering on $X^{**}$, where
$x_1>x_2>\dots$. We can also assume, by Shirshov algorithm, that
with the ordering on $[Y]X^{**}$ defined as before,
$S'=S^l\cup\{gx|g\in R,\ x\in X\}$ is a Gr\"{o}bner-Shirshov basis
in ${k[Y]}(X)$.

Let $B={k[Y]}(X,a,b|S_1\}$ where $S_1$ consists of
\begin{eqnarray*}
&&f_1=S^l,\\
&&f_2=\{gx|g\in R,\ x\in X\},\\
&& f_3=\{a(b^i)-x_i|i=1,2,\dots\},\\
&& f_4=\{ga|g\in R\},\\
&& f_5= \{gb|g\in R\}.
\end{eqnarray*}

Clearly, $B$ is a $K$-algebra generated by $a,b$. Thus, to prove the
theorem, by using our Theorem 2.6, it suffices to show that with the
ordering on $[Y](X\cup\{a,b\})^{**}$ as before, where $a>b>x_i,\
i=1,2,\dots$,
$S_1$ is a Gr\"{o}bner-Shirshov basis in ${k[Y]}(X,a,b)$.

Denote by $(i\wedge j)_{w_{ij}}$ the composition of the type $f_i$
and type $f_j$ with respect to the ambiguity $w_{ij}$. Since $S'$ is
a Gr\"{o}bner-Shirshov basis in ${k[Y]}(X)$, we need only to check
all compositions related to the following ambiguities $w_{ij}$:

$1\wedge4,\   \  w_{14}=L(\bar{f}^Y,
\bar{g})(z_1(\bar{f}^X)z_2az_3)$;

$1\wedge 5,\   \  w_{15}=L(\bar{f}^Y,
\bar{g})(z_1(\bar{f}^X)z_2bz_3)$;


$2\wedge 4,\   \  w_{24}=L(\bar{g'}, \bar{g})(z_1xz_2az_3)$;

$2\wedge 5,\   \  w_{25}=L(\bar{g'}, \bar{g})(z_1xz_2bz_3)$;

$3\wedge 4 ,\  \  w_{34}=\bar{g}a(b^i)$;

$3\wedge 5 ,\ \  w_{35}=\bar{g}a(b^i)$;

$4\wedge 1,\   \  w_{41}=L(\bar{g},
\bar{f}^Y)(z_1az_2(\bar{f}^X)z_3)$;

$4\wedge 2,\   \  w_{42}=L(\bar{g}, \bar{g'})(z_1az_2xz_3)$;

$4\wedge 4, \ \  w_{44}=L(\overline{g_1},\overline{g_2})a$;

$4\wedge 5,\   \  w_{45}=L(\bar{g}, \bar{g'})(z_1az_2bz_3)$;

$5\wedge 1,\   \  w_{51}=L(\bar{g},
\bar{f}^Y)(z_1bz_2(\bar{f}^X)z_3)$;

$5\wedge 2,\   \  w_{52}=L(\bar{g}, \bar{g'})(z_1bz_2xz_3)$;

$5\wedge 4,\   \  w_{54}=L(\bar{g}, \bar{g'})(z_1bz_2az_3)$;

$5\wedge 5, \ \  w_{55}=L(\overline{g_1},\overline{g_2})b$;

\noindent where $g,g', g_1, g_2\in R$, $f\in S^l$, $z_1,z_2,z_3\in
(X\cup\{a,b\})^*$ and $(z_1v_1z_2v_2z_3)$ is some bracketing.

Now, we prove that all the compositions are trivial.

$1\wedge 4 ,\  \   w_{14}=L(\bar{f}^Y,
\bar{g})(z_1(\bar{f}^X)z_2az_3)$, where $f\in S^l,\ g\in R$.

We can write $\bar{f}^X=(uxv)$, where $u,v\in X^*$. Since
$S'=\{S^l,Rx, x\in X\}$ is a Gr\"{o}bner-Shirshov basis in
${k[Y]}(X)$, we have $(f,gx)_{w}=\sum\alpha_iu_i|_{_{s_i}}$, where
$w=L(\bar{f}^Y, \bar{g})\bar{f}^X$, each $\alpha_i\in k,\ s_i\in
S',\ u_i\in [Y]X^{**}$ and $w>u_i|_{_{\overline{s_i}}}$. Then
\begin{eqnarray*}
(1,4)_{w_{14}}\
&=&\frac{L}{\bar{f}^Y}(z_1fz_2az_3)-\frac{L}{\bar{g}}(z_1(\bar{f}^X)z_2gaz_3)\\
&=&\frac{L}{\bar{f}^Y}(z_1fz_2az_3)-\frac{L}{\bar{g}}(z_1(ugxv)z_2az_3)+
\frac{L}{\bar{g}}(z_1(ugxv)z_2az_3)-\frac{L}{\bar{g}}(z_1(\bar{f}^X)z_2gaz_3)\\
&=&(z_1(\frac{L}{\bar{f}^Y}f-\frac{L}{\bar{g}}(ugxv))z_2az_3)+\frac{L}{\bar{g}}
g((z_1(uxv)z_2az_3)-(z_1(\bar{f}^X)z_2az_3))\\
&=&(z_1(f,gx)_{w}z_2az_3)+\frac{L}{\bar{g}}g((z_1(\bar{f}^X)z_2az_3)-(z_1(\bar{f}^X)z_2az_3))\\
&=&\sum\alpha_i(z_1u_i|_{_{s_i}}z_2az_3)\\
 &\equiv&0   ~~~~~mod(S_1,w_{14}).
\end{eqnarray*}
Similarly, $(1,5)_{w_{15}}\equiv0,\ (4,1)_{w_{41}}\equiv0,\ (5,1)_{w_{51}}\equiv0$.\\

$2\wedge 4,\   \  w_{24}=L(\bar{g'}, \bar{g})(z_1xz_2az_3)$, where
$g,g'\in R$.

If $|\bar{g'}|+|\bar{g}|>|L|$, then since $R$ is a
Gr\"{o}bner-Shirshov basis in $k[Y]$, $(g',
g)_w=(\frac{L}{\bar{g'}}g'-\frac{L}{\bar{g}}g)=\sum \alpha_iu_ih_i$,
where $w=L(\bar{g'},\ \bar{g})$, each $\alpha_i\in k, u_i\in [Y],\
h_i\in R$ and $w>u_i\overline{h_i}$. Thus
\begin{eqnarray*}
(2,4)_{w_{24}}&=&\frac{L}{\bar{g'}}(z_1g'xz_2az_3)-\frac{L}{\bar{g}}(z_1xz_2gaz_3)\\
&=&(\frac{L}{\bar{g'}}g'-\frac{L}{\bar{g}}g)(z_1xz_2az_3)\\
&=&\sum \alpha_iu_ih_i(z_1xz_2az_3)\\
&=&\sum \alpha_iu_i(z_1xz_2h_iaz_3)\\
&\equiv&0   ~~~~~mod(S_1,w_{24}).
\end{eqnarray*}.

Similarly, $(2,5)_{w_{25}}\equiv0,\ (4,2)_{w_{42}}\equiv0,\
(4,5)_{w_{45}}\equiv0,\
(5,2)_{w_{52}}\equiv0 $ and $(5,4)_{w_{54}}\equiv0$.\\

$3\wedge 4 ,\  \  w_{34}=\bar{g}a(b^i)$, where $ g\in R$.

Let $g=\bar{g}+r\in R$. Then
\begin{eqnarray*}
(3,4)_{w_{34}}\
&=&-\bar{g}x_i-ra(b^i)\\
&\equiv&-\bar{g}x_i-rx_i\\
&\equiv&gx_i\\
&\equiv&0   ~~~~~mod(S_1,w_{34}).
\end{eqnarray*}

Similarly, $(3,5)_{w_{35}}\equiv0$.\\

$4\wedge 4, \ \  w_{44}=L(\overline{g_1},\overline{g_2})a$, where
$g_1,g_2\in R$.

If $|\bar{g_1}|+|\bar{g_2}|>|L|$, then since $R$ is a
Gr\"{o}bner-Shirshov basis in $k[Y]$, $(g_1,
g_2)_w=(\frac{L}{\bar{g_1}}g_1-\frac{L}{\bar{g_2}}g_2)=\sum
\alpha_iu_ih_i$, where $w=L(\bar{g_1}, \bar{g_2})$, each
$\alpha_i\in k, u_i\in [Y],\ h_i\in R$ and $w>u_i\overline{h_i}$.
Thus
\begin{eqnarray*}
(4,4)_{w_{44}}
&=&\frac{L}{\bar{g_1}}(g_1a)-\frac{L}{\bar{g_2}}(g_2a) \\
&=&(\frac{L}{\bar{g_1}}g_1-\frac{L}{\bar{g_2}}g_2)a\\
&=&\sum \alpha_iu_ih_ia\\
&\equiv&0   ~~~~~mod(S_1,w_{44}).
\end{eqnarray*}.

If $|\bar{g_1}|+|\bar{g_2}|=|L|$, then
\begin{eqnarray*}
(4,4)_{w_{44}}
&=&\frac{L}{\bar{g_1}}(g_1a)-\frac{L}{\bar{g_2}}(g_2a) \\
&=&(\bar{g_2}g_1-\bar{g_1}g_2)a\\
&\equiv&((g_1-\bar{g_1})g_2-(g_2-\bar{g_2})g_1)a\\
&\equiv&0   ~~~~~mod(S_1,w_{44}).
\end{eqnarray*}

Similarly, $(5, 5)_{w_{55}}\equiv 0$.

Now we have proved that $S_1$ is a Gr\"{o}bner-Shirshov basis in
${k[Y]}(X,a, b)$.

The proof is complete. \end{proof}

A special case of Theorem 3.1 is the following corollary.

 \begin{Corollary}
Every countably generated non-associative algebra over a free
commutative algebra  can be embedded into a two-generated
non-associative algebra over a free commutative algebra.
 \end{Corollary}
 \ \

\noindent{\bf Acknowledgement}. The authors would like to express
their deepest gratitude to Professor L.A. Bokut for his kind
guidance, useful discussions and enthusiastic encouragement.

\end{document}